\documentclass[11pt,reqno]{amsart}

\usepackage{amssymb}
\usepackage{amsthm}
\usepackage{amsmath}
\usepackage{a4wide}
\usepackage{latexsym}
\usepackage{mathrsfs}
\usepackage{euscript}
\usepackage[mathx]{mathabx}
\usepackage[v2,tips]{xy}
\usepackage{float}
\usepackage[T1]{fontenc}
\usepackage[latin1]{inputenc}

\newtheorem{theorem}{Theorem}[section]

\newtheorem{proposition}[theorem]{Proposition}

\theoremstyle{definition}

\newtheorem{remark}[theorem]{Remark}

\numberwithin{equation}{section}

\newcounter{smallromans}

\newenvironment{romanenumerate}
{\begin{list}{{\normalfont\textrm{(\roman{smallromans})}}}%
    {\usecounter{smallromans}\setlength{\itemindent}{0cm}%
      \setlength{\leftmargin}{5.5ex}\setlength{\labelwidth}{5.5ex}%
      \setlength{\topsep}{0.2ex}\setlength{\partopsep}{0ex}%
      \setlength{\itemsep}{0.2ex}}}%
  {\end{list}}

\newcommand{\romanref}[1]{{\normalfont\textrm{(\ref{#1})}}}

\newcounter{smallalphs}

  {\end{list}}

\renewcommand{\le}{\ensuremath{\leqslant}}

\newcommand{\N}{\mathbb{N}}
\newcommand{\Z}{\mathbb{Z}}

\newcommand{\R}{\mathbb{R}}
\newcommand{\C}{\mathbb{C}}
\newcommand{\K}{\mathbb{K}}

\renewcommand{\phi}{\ensuremath{\varphi}}
\renewcommand{\epsilon}{\ensuremath{\varepsilon}}

\newcommand{\smashw}[2][l]{{\text{\makebox[0pt][#1]{$#2$}}}}

\begin{document}
\title[Splittings of extensions and homological bidimension of
  $\mathscr{B}(E)$]{Splittings of extensions and homological
  bi\-dimen\-sion of the algebra of bounded operators on a Banach
  space}
% Scindages d'extensions et bidimension homologique de l'alg\`{e}bre
% des op\'{e}rateurs born\'{e}s sur un espace de Banach.
\dedicatory{In memoriam: Uffe Haagerup (1949--2015)} 
\subjclass[2010]%
{Primary
  46H40, % Automatic continuity
  46M18, % Homological methods (exact sequences, right inverses, lifting, etc.)
  47L10; % Algebras of operators on Banach spaces and other topological linear spaces
   Secondary 16E10, %   Homological dimension
  16S70, % Extensions of rings by ideals
46H10} % Ideals and subalgebras
\author[N.~J.~Laustsen]{Niels Jakob Laustsen} 
\author[R.~Skillicorn]{Richard Skillicorn} 
\address{Department of Mathematics and Statistics, Fylde College,
  Lancaster University, Lancaster LA1 4YF, United Kingdom.}
\email{n.laustsen@lancaster.ac.uk,\ r.skillicorn@lancaster.ac.uk}
\keywords{Bounded, linear operator; Banach space; Banach algebra;
  short-exact sequence; algebraic splitting; strong splitting;
  singular extension; admissible extension; pullback; second
  continuous Hochschild cohomology group; homological bidimension.}
%quasi-reflexivity;
\begin{abstract}
  We show that there exists a Banach space~$E$ such that: 
  \begin{itemize}
  \item the Banach algebra~$\mathscr{B}(E)$ of bounded, linear
    operators on~$E$ has a singular extension which splits
    algebraically, but it does not split strongly;
  \item the homological bidimension of~$\mathscr{B}(E)$ is at least
    two.
  \end{itemize}
  The first of these conclusions solves a natural  problem left open by
  Bade, Dales, and Lykova (\emph{Mem.\ Amer.\ Math.\ Soc.}~1999),
  while the second answers a question of Helemskii. The Ba\-nach
  space~$E$ that we use was originally introduced by Read
  (\emph{J.~London Math.\ Soc.}~1989).

  Nous d\'{e}montrons qu'il existe un espace de Banach tel que:
  \begin{itemize}
  \item l'alg\`{e}bre de Banach~$\mathscr{B}(E)$ des op\'{e}rateurs
    lin\'{e}aires born\'{e}s sur~$E$ a une extension singuli\`{e}re
    qui scinde alg\'{e}briquement mais qui ne scinde pas fortement;
  \item la bidimension homologique de~$\mathscr{B}(E)$ est au moins
    deux.
  \end{itemize}
  La premi\`{e}re de ces conclusions compl\`{e}te les r\'{e}sultats de
  Bade, Dales et Lykova (\emph{Mem.\ Amer.\ Math.\ Soc.}~1999), et la
  seconde r\'{e}pond \`{a} une question de Helemskii.  L'espace de
  Banach~$E$ a \'{e}t\'{e} introduit initialement par Read
  (\emph{J.~London Math.\ Soc.}~1989).
\end{abstract}
\maketitle
\section{Introduction and statement of results}%
\label{section1}
\noindent 
By an \emph{extension} of a Banach algebra~$\mathscr{B}$, we
understand a short-exact sequence of the form
\begin{equation}\label{shortexactseq} \spreaddiagramcolumns{2ex}%
    \xymatrix{\{0\}\ar[r] & \ker\phi\ar[r] &
      \mathscr{A}\ar^-{\displaystyle{\phi}}[r] & \mathscr{B}\ar[r] &
      \{0\},}
\end{equation} 
where~$\mathscr{A}$ is a Banach algebra
and~$\phi\colon\mathscr{A}\to\mathscr{B}$ is a continuous, surjective
algebra homomorphism. The ex\-ten\-sion \emph{splits algebraically}
(respectively, \emph{splits strongly}) if there is an algebra
homo\-mor\-phism (respectively, a continuous algebra homomorphism)
$\rho\colon\mathscr{B}\to\mathscr{A}$ which is a right in\-verse
of~$\phi$, in the sense that $\phi\circ\rho$ is the identity map
on~$\mathscr{B}$.  We say that~\eqref{shortexactseq} is
\emph{admissible} if~$\phi$ has a 
%bounded, linear 
right in\-verse which is bounded and linear, or, equivalently, if
$\ker\phi$ is complemented in~$\mathscr{A}$ as a Banach space.  Every
extension which splits strongly is obviously admissible. The
extension~\eqref{shortexactseq} is \emph{singular} if~$\ker\phi$ has
trivial multiplication, in the sense that $ab=0$ whenever
$a,b\in\ker\phi$.

Bade, Dales, and Lykova~\cite{bdl} carried out a comprehensive study
of extensions of Banach algebras, focusing in particular on the
following question:
\begin{quote}
\textsl{For which (classes of) Banach algebras~$\mathscr{B}$ is it
  true that every extension of the form}~\eqref{shortexactseq}
\textsl{which splits algebraically also splits strongly?}
\end{quote}
This question can be viewed as a variation on the theme of automatic
continuity.  Of course, its answer is positive whenever the Banach
algebra~$\mathscr{B}$ has the property that \textsl{every} alge\-bra
homomorphism from~$\mathscr{B}$ into a Banach algebra is continuous. A
classical theorem of Johnson~\cite{johnson} states that the Banach
algebra $\mathscr{B} = \mathscr{B}(E)$ of all bounded operators on a
Banach space~$E$ has this property whenever~$E$ is isomorphic to its
Cartesian square $E\oplus E$.

Johnson's result, how\-ever, does not extend to all Banach spaces
because Read~\cite{read} has constructed a Banach space~$E_{\text{R}}$
such that there exists a discontinuous derivation (and hence a
discontinuous algebra homomorphism) from~$\mathscr{B}(E_{\text{R}})$.
Dales, Loy, and Willis~\cite{dlw} have subsequently given an example
of a Banach space~$E_{\text{DLW}}$ such that all derivations
from~$\mathscr{B}(E_{\text{DLW}})$ are continuous, but under the
assump\-tion of the Continuum Hypothesis,
$\mathscr{B}(E_{\text{DLW}})$~admits a discontinuous algebra
homo\-mor\-phism into a Banach algebra.

These results still leave open the above question of Bade, Dales, and
Lykova in the case of~$\mathscr{B}(E)$ for a general Banach space~$E$:
is it true that every extension of~$\mathscr{B}(E)$ which splits
algebraically also splits strongly?  Our first result answers
this question in the negative.

\begin{theorem}\label{mainthm}
There exists a continuous, surjective algebra homomorphism~$\phi$
from a unital Banach algebra~$\mathscr{A}$
onto~$\mathscr{B}(E_{\normalfont{\text{R}}})$,
where~$E_{\normalfont{\text{R}}}$ denotes the above-mentioned Banach
space of Read, such that the
extension
\[ \spreaddiagramcolumns{2ex}\xymatrix{\{0\}\ar[r] &
   \ker\phi\ar[r] & \mathscr{A}\ar^-{\displaystyle{\phi}}[r] &
   \mathscr{B}(E_{\normalfont{\text{R}}})\ar[r] & \{0\}} \] is
singular and splits algebraically, but it is not admissible, and so it
does not split strongly.
\end{theorem}

We do not know whether Read's Banach space~$E_{\normalfont{\text{R}}}$
is essential for this result. Due to the dearth of examples of Banach
spaces~$E$ for which~$\mathscr{B}(E)$ admits a discontinuous
homomorphism into a Banach algebra, $E = E_{\normalfont{\text{R}}}$
was the most obvious place to start our investigations, and it led to
the answer that we were looking for.

Before we can state our second result, we require some more
terminology. Let~$n\in\N_0$. A Banach algebra~$\mathscr{B}$ has
\emph{homological bidimension at least~$n$} if there exists a Banach
$\mathscr{B}$\nobreakdash-bi\-mod\-ule~$X$ such that the
$n^{\text{th}}$ continuous Hochschild cohomology group
$\mathscr{H}^n(\mathscr{B}, X)$ of~$\mathscr{B}$ with coefficients
in~$X$ is non-zero.  This notion is the topological counterpart of a
long established notion in pure algebra. It was introduced by
Helemskii, who, together with his students, has studied it for many
classes of Banach algebras; see~\cite{helemskii} for an overview.  

A Banach algebra~$\mathscr{B}$ which has homological bidimension zero
(so that $\mathscr{H}^1(\mathscr{B}, X) = \{0\}$ for every Banach
$\mathscr{B}$\nobreakdash-bi\-mod\-ule~$X$) is \emph{contractible}.
It is conjectured that a Banach algebra is contractible (if and) only
if it is finite-dimensional and semi\-simple.

If true, this conjecture would imply that the Banach
algebra~$\mathscr{B}(E)$ has homological bidimension at least one for
every infinite-dimensional Banach space~$E$; see
\cite[Proposition~5.1]{johnsonSymmAm} for a strong partial result in
this direction.  It appears to be unknown if a Banach algebra of the
form~$\mathscr{B}(E)$ for a (necessarily in\-finite-di\-men\-sional)
Banach space~$E$ can have homological bidimension at least two, a
problem that goes back to Helemskii's seminar at Moscow State
University. In the case where~$E$ is a Hilbert space, this problem is
stated explicitly as \cite[Problem~21]{heProblems}; see also
\cite[p.~27]{bdl}.  We shall show that the homological bidimension
of~$\mathscr{B}(E)$ can be two or greater, again using the
above-mentioned Banach space~$E_{\normalfont{\text{R}}}$ of Read.

\begin{theorem}\label{mainthmH2BER}
There exist a one-dimensional Banach
$\mathscr{B}(E_{\text{\normalfont{R}}})$-bimodule~$X$ and a linear
in\-jec\-tion from the Banach algebra~$\mathscr{B}(\ell_2(\N))$ of
bounded operators on the Hilbert space~$\ell_2(\N)$ into the second
continuous Hochschild cohomology group
of~$\mathscr{B}(E_{\text{\normalfont{R}}})$ with coefficients in~$X$.
Hence~$\mathscr{B}(E_{\text{\normalfont{R}}})$ has homological
bidimension at least two.
\end{theorem}

\section{Proofs of Theorems~\ref{mainthm} and~\ref{mainthmH2BER}}
\noindent
The proofs of Theorems~\ref{mainthm} and~\ref{mainthmH2BER} both rely
on a strengthening of the main technical result in Read's paper, as it
is stated in \cite[Section~4]{read}. This strengthening involves two
further pieces of notation.  First, we denote
by~$\mathscr{W}(E_{\text{R}})$ the ideal of weakly compact operators
on the Banach space~$E_{\text{R}}$. Second, we endow the Hilbert space
$\ell_2(\N)$ with the trivial multiplication and write
$\ell_2(\N)^{\sim}$ for its unitization; that is, $\ell_2(\N)^{\sim} =
\ell_2(\N)\oplus\K 1$ as a vector space (where~$\K$ denotes the scalar
field, either~$\R$ or~$\C$, and~$1$ is the formal identity that we
adjoin), while the product and the norm on~$\ell_2(\N)^{\sim}$ are
given by
\[ (x + \lambda
1)(y+\mu 1) =
\lambda y+\mu x+ \lambda\mu 1\quad
\text{and}\quad \|x + \lambda 1\| = \|x\|
+ |\lambda|\qquad (x,y\in\ell_2(\N),\, \lambda,\mu\in\mathbb{K}). \]

\begin{theorem}\label{WEBEsplitexact} There exists a continuous, 
surjective algebra homomorphism~$\psi$ 
from~$\mathscr{B}(E_{\normalfont{\text{R}}})$
onto~$\ell_2(\N)^{\sim}$ with $\ker\psi =
\mathscr{W}(E_{\normalfont{\text{R}}})$ such that the extension 
\[ \spreaddiagramcolumns{2ex}\xymatrix{\{0\}\ar[r] &
    \mathscr{W}(E_{\normalfont{\text{R}}})\ar[r] &
    \mathscr{B}(E_{\normalfont{\text{R}}})\ar^-{\displaystyle{\psi}}[r]
    & \ell_2(\N)^{\sim}\ar[r] & \{0\}}
\] splits strongly.
\end{theorem}

The proof of this result relies on a careful analysis of Read's
construction; full details will appear in \cite{LS}.

In order to prove Theorem~\ref{mainthm}, we require another
tool, namely the \emph{pullback} of a diagram of the form
\begin{equation}\label{prepullbackdiagram}
  \spreaddiagramcolumns{2ex}%
\xymatrix{\mathscr{A}\ar^-{\displaystyle{\alpha}}[r] & 
 \mathscr{C} & \mathscr{B},\ar_-{\displaystyle{\beta}}[l]}
\end{equation}
where~$\mathscr{A}$, $\mathscr{B}$, and~$\mathscr{C}$ are Banach
algebras, and $\alpha\colon\mathscr{A}\to\mathscr{C}$ and
$\beta\colon\mathscr{B}\to\mathscr{C}$ are continuous algebra
homomorphisms. We can define the pullback of this diagram explicitly
by the formula
\begin{equation}\label{expldefnpullback} 
\mathscr{D} = \{(a,b)\in\mathscr{A}\oplus\mathscr{B} : \alpha(a) =
\beta(b) \}, \end{equation} where~$\mathscr{A}\oplus\mathscr{B}$
denotes the direct sum of the Banach algebras~$\mathscr{A}$ and~$\mathscr{B}$.  Being a
closed subalgebra of~$\mathscr{A}\oplus\mathscr{B}$, $\mathscr{D}$~is
a Banach algebra in its own right.  Let
\begin{equation}\label{defngammadelta}
 \gamma\colon\ (a,b)\mapsto a,\quad
 \mathscr{D}\to\mathscr{A},\qquad\text{and}\qquad
 \delta\colon\ (a,b)\mapsto b,\quad
 \mathscr{D}\to\mathscr{B}, \end{equation} be the restrictions
to~$\mathscr{D}$ of the coordinate projections. Then
$\alpha\circ\gamma = \beta\circ\delta$, and it can be shown
that~$\mathscr{D}$, together with the continuous algebra
homomorphisms~$\gamma$ and~$\delta$, has the following universal
property, so that they form a pullback of~\eqref{prepullbackdiagram} in
the categorical sense: for every Banach algebra~$\mathscr{E}$ and each
pair $\xi\colon\mathscr{E}\to\mathscr{A}$ and
$\eta\colon\mathscr{E}\to\mathscr{B}$ of continuous algebra
homomorphisms satisfying $\alpha\circ\xi = \beta\circ\eta$, there is a
unique continuous algebra homomorphism
$\theta\colon\mathscr{E}\to\mathscr{D}$ such that the diagram
\begin{equation*}
  \spreaddiagramcolumns{.5ex}\spreaddiagramrows{-.5ex}% 
  \xymatrix{%
  \mathscr{E}\ar^-{\displaystyle{\xi}}[rrrd]%
  \ar@{-->}^(0.63){\displaystyle{\theta}}[rd]%
  \ar_-{\displaystyle{\eta}}[rddd]\\ 
  &\mathscr{D}\ar_-{\displaystyle{\gamma}}[rr]%
  \ar^-{\displaystyle{\delta}}[dd]
  &&\mathscr{A}\ar^-{\displaystyle{\alpha}}[dd]\\ \\
  &\mathscr{B}\ar^-{\displaystyle{\beta}}[rr]
  &&\mathscr{C}}%end xymatrix
\end{equation*}
is commutative.

We now come to our key result, which establishes a connection between
extensions and pull\-backs.

\begin{proposition}\label{pullbacksandsplittingsofextensions}
Let~$\mathscr{A}$, $\mathscr{B}$, and~$\mathscr{C}$ be Banach algebras
such that there are extensions
\begin{equation}\label{ext1} \spreaddiagramcolumns{2ex}%
    \xymatrix{\{0\}\ar[r] & \ker\alpha\ar[r] &
      \mathscr{A}\ar^-{\displaystyle{\alpha}}[r] & \mathscr{C}\ar[r] &
      \{0\}}
\end{equation} 
and 
\begin{equation}\label{ext2} \spreaddiagramcolumns{2ex}%
    \xymatrix{\{0\}\ar[r] & \ker\beta\ar[r] &
      \mathscr{B}\ar^-{\displaystyle{\beta}}[r] & \mathscr{C}\ar[r] &
      \{0\}\smashw{,}}
\end{equation}
and define $\mathscr{D}$, $\gamma$, and $\delta$
by~\eqref{expldefnpullback} and~\eqref{defngammadelta},
above. Then~$\delta$ is surjective, and the following statements
concerning the extension
\begin{equation}\label{ext3} \spreaddiagramcolumns{2ex}%
    \xymatrix{\{0\}\ar[r] & \ker\delta\ar[r] &
      \mathscr{D}\ar^-{\displaystyle{\delta}}[r] & \mathscr{B}\ar[r] &
      \{0\}}
\end{equation}
 hold true:
\begin{romanenumerate}
\item\label{pullbacksandsplittingsofextensions1} \eqref{ext3} is
  singular if and only if~\eqref{ext1} is singular.
\item\label{pullbacksandsplittingsofextensions2} Suppose
  that~\eqref{ext2} splits strongly (respectively, splits
  algebraically, is admissible).  Then \eqref{ext3} splits strongly
  (respectively, splits algebraically, is admissible) if and only
  if~\eqref{ext1} splits strongly (respectively, splits algebraically,
  is admissible).
\end{romanenumerate}
\end{proposition}

\begin{proof}
The surjectivity of~$\alpha$ implies that~$\delta$ is surjective, so
that~\eqref{ext3} is indeed an extension.

\romanref{pullbacksandsplittingsofextensions1}.  The restriction
of~$\gamma$ to~$\ker\delta$ is an isomorphism onto~$\ker\alpha$, and
the conclusion follows.

\romanref{pullbacksandsplittingsofextensions2}. Let
$\rho\colon\mathscr{C}\to\mathscr{B}$ be a continuous algebra
homomorphism which is a right inverse of~$\beta$. 

$\Rightarrow$. Suppose that $\tau\colon\mathscr{B}\to\mathscr{D}$ is a
continuous algebra homomorphism which is a right inverse
of~$\delta$. Then a direct calculation shows that the continuous
algebra homomorphism $\gamma\circ\tau\circ\rho$ is a right inverse
of~$\alpha$, so that~\eqref{ext1} splits strongly.

$\Leftarrow$. Suppose that $\sigma\colon\mathscr{C}\to\mathscr{A}$ is
a continuous algebra homomorphism which is a right inverse
of~$\alpha$. Then, setting $\tau(b) = (\sigma(\beta(b)), b)$ for each
$b\in\mathscr{B}$, we obtain a continuous algebra homomorphism
\mbox{$\tau\colon\mathscr{B}\to\mathscr{D}$}. The definition
of~$\delta$ implies that~$\tau$ is a right inverse of~$\delta$, and
hence~\eqref{ext3} splits strongly.

The proof just given applies equally to establish the other two cases.
\end{proof}

\begin{proof}[Proof of  Theorem~{\normalfont{\ref{mainthm}}}]
Our aim is to apply
Proposition~\ref{pullbacksandsplittingsofextensions} with $\mathscr{B}
= \mathscr{B}(E_{\text{R}})$, $\mathscr{C} = \ell_2(\N)^{\sim}$, and
$\beta = \psi$. Theorem~\ref{WEBEsplitexact} shows that, for these
choices, we have an extension of the form~\eqref{ext2} which splits
strongly.  

Let $q\colon\ell_1(\N)\to\ell_2(\N)$ be a bounded, linear
surjection. Then $\ker q$ is not complemented in~$\ell_1(\N)$ because
no (complemented) subspace of~$\ell_1(\N)$ is isomorphic
to~$\ell_2(\N)$. Equip~$\ell_1(\N)$ with the trivial product, let
$\mathscr{A} = \ell_1(\N)\oplus\K1$ be its unitization (defined
analogously to the unitization of~$\ell_2(\N)$, above), and define
$\alpha\colon\mathscr{A}\to \mathscr{C}$ by $\alpha(x+\lambda 1) =
q(x)+\lambda 1$ for $x\in\ell_1(\N)$ and $\lambda\in\K$. Then~$\alpha$
is a continuous, surjective algebra homomorphism with kernel~$\ker q$,
which is uncomplemented in~$\ell_1(\N)$ and hence in~$\mathscr{A}$, so
that we have a singular, non-admissible ex\-ten\-sion of the
form~\eqref{ext1}.

Being surjective, $q$ has a linear right inverse
$\rho\colon\ell_2(\N)\to\ell_1(\N)$, which is multiplicative be\-cause
$\ell_1(\N)$ and $\ell_2(\N)$ both have the trivial
product. Extend~$\rho$ to a linear map between the
uni\-ti\-za\-tions~$\mathscr{C}$ and~$\mathscr{A}$ by making it unital. Then
it is an algebra homomorphism which is a right inverse of~$\alpha$, so
that the extension~\eqref{ext1} splits algebraically. Hence
Proposition~\ref{pullbacksandsplittingsofextensions} pro\-duces a
singular extension~\eqref{ext3} of $\mathscr{B} =
\mathscr{B}(E_{\text{R}})$ which splits algebraically, but is not
admissible.
\end{proof}

Before we proceed to prove Theorem~\ref{mainthmH2BER}, let us recall
the formal definition of the second continuous Hochschild cohomology
group of a Banach algebra~$\mathscr{B}$ with coefficients in a Banach
$\mathscr{B}$\nobreakdash-bi\-module~$X$.  A \emph{$2$-cocycle} is a
bilinear map $\Upsilon\colon\mathscr{B}\times\mathscr{B}\to X$ which
satisfies
\[ a\cdot \Upsilon(b,c) - \Upsilon(ab,c) + \Upsilon(a,bc) - 
\Upsilon(a,b)\cdot c = 0\qquad (a,b,c\in\mathscr{B}). \] The set
$\mathscr{Z}^2(\mathscr{B}, X)$ of continuous $2$-cocycles forms a
closed subspace of the Banach space of continuous, bilinear maps
from~$\mathscr{B}\times\mathscr{B}$ into~$X$.  Each bounded, linear
map $\Omega\colon\mathscr{B}\to X$ induces a continuous $2$-cocycle by
the definition
\begin{equation}\label{coboundDefn}
\delta^1\Omega\colon\ (a,b)\mapsto a\cdot (\Omega b) -
\Omega(ab) + (\Omega a)\cdot b,\quad
\mathscr{B}\times\mathscr{B}\to X. \end{equation}
The $2$-cocyles of this form are called \emph{$2$-coboundaries;}
they form a (not necessarily closed) subspace
$\mathscr{N}^2(\mathscr{B}, X)$ of~$\mathscr{Z}^2(\mathscr{B}, X)$,
and so the quotient
\[ \mathscr{H}^2(\mathscr{B}, X) := 
\mathscr{Z}^2(\mathscr{B}, X)/\mathscr{N}^2(\mathscr{B}, X) \] is a
vector space, which is the \emph{second continuous Hochschild
  cohomology group} of~$\mathscr{B}$ with \emph{coefficients} in~$X$.

\begin{proof}[Proof of Theorem~{\normalfont{\ref{mainthmH2BER}}}]
By Theorem~\ref{WEBEsplitexact}, there are continuous algebra
homomorphisms
\[ \psi\colon 
\mathscr{B}(E_{\text{\normalfont{R}}})\to\ell_2(\N)^{\sim}
\qquad\text{and}\qquad \rho\colon
\ell_2(\N)^{\sim}\to\mathscr{B}(E_{\text{\normalfont{R}}}) \] such
that~$\rho$ is a right inverse of~$\psi$. The definition of the
unitization implies that we can find maps
$\psi_0\colon\mathscr{B}(E_{\text{\normalfont{R}}})\to\ell_2(\N)$ and
$\theta\colon\mathscr{B}(E_{\text{\normalfont{R}}})\to\mathbb{K}$ such
that \[ \psi(T) = \psi_0(T) + \theta(T)1\qquad
(T\in\mathscr{B}(E_{\text{\normalfont{R}}})). \]
We see that~$\theta$ is a continuous algebra
homomorphism, and $X = \mathbb{K}$ is a one-dimensional Banach
$\mathscr{B}(E_{\text{\normalfont{R}}})$-bimodule with respect to the
operations
\begin{equation}\label{onedimbimodDefns}
T\cdot \lambda = \theta(T)\lambda\qquad\text{and}\qquad \lambda \cdot
T = \theta(T)\lambda\qquad
(T\in\mathscr{B}(E_{\text{\normalfont{R}}}),\,\lambda\in
X). \end{equation} Moreover, $\psi_0$ is bounded and linear.
Consequently, for each $U\in\mathscr{B}(\ell_2(\N))$, we can define a
continuous, bilinear map $\Upsilon_U\colon
\mathscr{B}(E_{\text{\normalfont{R}}})\times
\mathscr{B}(E_{\text{\normalfont{R}}})\to X$ by
\begin{equation}\label{DefnUpsilonU} 
\Upsilon_U(S,T) = \langle U(\psi_0(S)),\psi_0(T)\rangle\qquad
(S,T\in\mathscr{B}(E_{\text{\normalfont{R}}})), \end{equation} where
$\langle\cdot,\cdot\rangle$ denotes the usual Banach-space duality
bracket on~$\ell_2(\N)$, that is, $\langle x,y\rangle =
\sum_{n=1}^\infty x_ny_n$ for $x=(x_n)$ and $y=(y_n)$
in~$\ell_2(\N)$. The map~$\psi_0$ is not multiplicative; more
precisely, since~$\ell_2(\N)$ has trivial multiplication, we have
\[ \psi_0(ST) = \theta(S)\psi_0(T) + \theta(T)\psi_0(S)\qquad
(S,T\in\mathscr{B}(E_{\text{\normalfont{R}}})). \] A straightforward
verification based on this identity shows that~$\Upsilon_U$ is a
$2$-cocycle.  Hence we have a map $\Upsilon\colon
U\mapsto\Upsilon_U,\,
\mathscr{B}(\ell_2(\N))\to\mathscr{Z}^2(\mathscr{B}(E_{\text{\normalfont{R}}}),
X)$, which is clearly linear. 
%It is injective because~$\psi_0$ is surjective.

Suppose that~$\Upsilon_U$ is a $2$-coboundary for some
$U\in\mathscr{B}(\ell_2(\N))$, so that $\Upsilon_U = \delta^1\Omega$
for some bounded, linear map
$\Omega\colon\mathscr{B}(E_{\text{\normalfont{R}}})\to X$. 
Since~$\rho$ is a right inverse of~$\psi$, we see
that $\psi_0(\rho(x)) = x$ and $\theta(\rho(x)) = 0$ for each $x\in\ell_2(\N)$. 
Combining these identities with the
definitions~\eqref{coboundDefn}--\eqref{DefnUpsilonU}, we obtain
\[ \langle Ux, y\rangle = 
(\delta^1\Omega)(\rho(x),\rho(y)) = \theta(\rho(x))\Omega(\rho(y)) -
\Omega(\rho(x)\rho(y)) + \theta(\rho(y))\Omega(\rho(x)) = 0 \] for
each $x,y\in\ell_2(\N)$ because $\rho(x)\rho(y) = \rho(xy) = 0$. This
shows that $ U= 0$, so that~$0$ is the only $2$-coboundary in the
image of~$\Upsilon$.  Hence the composition of~$\Upsilon$ with the
quotient map
from~$\mathscr{Z}^2(\mathscr{B}(E_{\text{\normalfont{R}}}),X)$
onto~$\mathscr{H}^2(\mathscr{B}(E_{\text{\normalfont{R}}}),X)$ is a
linear injection from~$\mathscr{B}(\ell_2(\N))$
into~$\mathscr{H}^2(\mathscr{B}(E_{\text{\normalfont{R}}}),X)$, and
the result follows.
\end{proof} 

\begin{remark}
There is an underlying connection between Theorems~\ref{mainthm}
and~\ref{mainthmH2BER}. To explain it, consider two
extensions~\eqref{ext1} and~\eqref{ext2} of a Banach
algebra~$\mathscr{C}$, where the former extension is singular and
admissible, but does not split strongly, while the latter splits
strongly. Then, by
Proposition~\ref{pullbacksandsplittingsofextensions}, we obtain a
singular, admissible extension~\eqref{ext3} of the Banach
algebra~$\mathscr{B}$, and this extension does not split
strongly. Hence a classical result of Johnson (see
\cite[Theorem~2.1]{johnsonWed}, or \cite[Corollary~I.1.11]{helemskii}
for an exposition) implies that $\ker\delta$ is a Banach
$\mathscr{B}$-bimodule and $\mathscr{H}^2(\mathscr{B},\ker\delta)$ is
non-zero, so that~$\mathscr{B}$ has homological bidimension at least
two.

To apply this result to $\mathscr{B} =
\mathscr{B}(E_{\text{\normalfont{R}}})$, we take $\mathscr{C} =
\ell_2(\N)^{\sim}$ and $\beta = \psi$ as in the proof of
Theorem~\ref{mainthm}, so that we have an extension of the
form~\eqref{ext2} which splits strongly by
Theorem~\ref{WEBEsplitexact}. Choose $U\in\mathscr{B}(\ell_2(\N))$
with $\|U\|\le1$, and turn the vector space $\K\oplus\ell_2(\N)$ into
a Banach algebra by endowing it with the product and the norm
\[  (\lambda,x)(\mu,y) = (\langle Ux,y\rangle, 
0)\qquad\text{and}\qquad \|(\lambda,x)\| = |\lambda|+ \|x\|\qquad
(x,y\in\ell_2(\N),\,\lambda,\mu\in\K). \] Denote by~$\mathscr{A}$ the
unitization of this Banach algebra, and
let~$\alpha\colon\mathscr{A}\to\mathscr{C}$ be the natural unital
projection. Then $\alpha$ is a continuous, surjective algebra
homomorphism, and we have a singular, admissible extension of the
form~\eqref{ext1}, which can be shown to split algebraically if and
only if it splits strongly, if and only if $U=0$. Thus, choosing~$U$
non-zero, we conclude that~$\mathscr{B}(E_{\text{\normalfont{R}}})$
has homological bidimension at least two.

A similar argument shows
that~$\mathscr{B}(E_{\text{\normalfont{DLW}}})$ has homological
bidimension at least two, where $E_{\text{\normalfont{DLW}}}$ denotes
the Banach space of Dales, Loy, and Willis studied in~\cite{dlw}.  To
see this, take $\mathscr{B} =
\mathscr{B}(E_{\text{\normalfont{DLW}}})$ and $\mathscr{C} =
\ell_\infty(\Z)$, and apply \cite[Theorem~3.11(i)]{bdl} to obtain a
singular, admissible extension of~$\ell_\infty(\Z)$ which does not
split strongly.
\end{remark}

\subsection*{Acknowledgements} We are grateful to 
Garth Dales (Lancaster) and Zinaida Lykova (Newcastle) for having
drawn our attention to the questions of whether every extension
of~$\mathscr{B}(E)$ which splits algebraically also splits strongly,
and whether~$\mathscr{B}(E)$ may have homological bidimension at least
two, respectively, and for their comments on preliminary versions of
our work. We would also like to thank Yemon Choi (Lancaster) for
several helpful discussions.

\bibliographystyle{amsplain}

\end{document}